\documentclass[12pt]{article}
\usepackage[T2A]{fontenc}
\usepackage[cp1251]{inputenc}
\usepackage[russian,english]{babel}
\usepackage{amsmath,amssymb,amsfonts,amsthm,bbding}

\theoremstyle{plain}
\newtheorem{theorem}{Theorem}

\newtheorem{corollary}[theorem]{Corollary}

\let\geq\geqslant
\let\leq\leqslant

\def\starr{\dot\times}

\begin{document}

\title{Restricted Product Sets under Unique Representability}
\author{F. Petrov}
\maketitle

\let\thefootnote\relax\footnote{
St. Petersburg Department of
V.~A.~Steklov Institute of Mathematics of
the Russian Academy of Sciences, St. Petersburg State University.
E-mail: fedyapetrov@gmail.com.
Supported by Russian Scientific Foundation grant 14-11-00581.}

\begin{abstract}
We prove some results of Kemperman--Scherk
 type for restricted product sets in multiplicative
groups of fields (in particular, for cyclic groups). The proofs
use polynomial method.
\end{abstract}

Recall a classical theorem by Scherk \cite{S} and
Kemperman \cite{K1,K2} (see history of this result in \cite{Lev}):

\begin{theorem}\label{KS} Let  $A,B$ be two finite subsets of some abelian group $G$.
Assume that some element $c\in A\cdot B$ has unique representation
as $c=ab$, $a\in A,b\in B$. Then  $|A\cdot B|\geqslant |A|+|B|-1$.
\end{theorem}

This is proved purely combinatorial methods,
that is, by by Dyson changes $(A,B)\rightarrow (A,xB)\rightarrow (A\cap xB,a\cup xB)$.

Consider now the simplest restricted version. 
We denote $A\starr B=\{ab:a\in A,b\in B,a\ne b\}$. If group $G$ is additive,
we use $\dotplus$.

We use polynomial method, in particular, the following useful corollary 
\cite{AF} of
Alon's Combinatorial Nullstellensatz \cite{Alon}:

\begin{theorem}\label{AlFu} If $X$, $Y$ are finite non-empty subsets of
the field $K$, and polynomial $f(x,y)$ equals to 0 on all but 1 points
of the grid $X\times Y$, then $\deg f\geq |X|+|Y|-2$.
\end{theorem}

This in turn follows from the following formula \cite{L,KP} expressing a coefficient of $x^{|A|-1}y^{|B|-1}$
at any polynomial $f(x,y)$, $\deg f(x,y)\leq |A|+|B|-2$:
\begin{equation}\label{formula}
[x^{|A|-1}y^{|B|-1}] f(x,y)=\sum_{x\in A,y\in B} \frac{f(t,s)}{\prod_{\tau\in A\setminus t} (t-\tau)
\prod_{\xi\in B\setminus s}(s-\xi)}.
\end{equation}
Indeed, if $\deg f<|X|+|Y|-2$ in  Theorem \ref{AlFu}, then 
apply \eqref{formula} for the polynomial $f(x,y)$ and sets 
$X=A$, $Y=B$: LHS of \eqref{formula} equals to 0, 
while in RHS there exists a unique
non-zero summand.

Further we need this formula itself. The following result is essentially contained in
\cite{PS} (the formal difference is that we consider number of representations $c=a+b, a\ne b$,
while Pan and Sun consider all representations $c=a+b$.)

\begin{theorem}\label{additive}
Let  $A,B$ be two finite subsets of additive group of the field.
Assume that some element $c\in A\dotplus B$ has unique representation
as $c=a+b$, $a\in A,b\in B$, $a\ne b$. Then  $|A\dotplus B|\geqslant |A|+|B|-2$.
\end{theorem}

\begin{proof} Assume the contrary. Consider $|A\dotplus B|$ lines: $x=y$ and $x+y=\gamma$, 
$\gamma\in (A\starr B)\setminus c$. They cover all points of $A\times B$ except
$(a,b)$. Hence $|A\dotplus B|\geqslant |A|+|B|-2$
by Alon--F\"uredi Theorem \ref{AlFu}.
\end{proof}

\begin{corollary} Let  $A$ be a finite subset of additive group of the field.
Assume that some element $c\in A\dotplus A$ has exactly two (symmetric to
each other) representations
as $c=a+b=b+a$, $a,b\in A$, $a\ne b$. Then  $|A\dotplus A|\geqslant 2|A|-3$.
\end{corollary}

\begin{proof}
Denote $B=A\setminus a$ and apply Theorem \ref{additive}.
\end{proof}

Now consider the case of multiplicative group of the field. 
This is proved to be useful approach for studying cyclic
groups, in the context of restricted sumsets/product sets
it originated from \cite{K}.

The estimate becomes slightly worse:

\begin{theorem}\label{multiplicative} 
Let  $A,B$ be two finite subsets of the multiplicative group of a field.
Assume that some element $c\in A\starr B$ has unique representation
as $c=ab$, $a\in A,b\in B$, $a\ne b$. Then  $|A\starr B|\geqslant |A|+|B|-3$.
\end{theorem}

\begin{proof} Consider the set $\setminus (a,b^{-1})$.
A polynomial $(xy-1)\prod_{\gamma\in (A\starr B)\setminus c} (x-\gamma y)$
vanishes on all points of $A\times B^{-1}$ except $(a,b^{-1})$. Hence its degree
$|C|+1$ is not less than $|A|+|B|-2$. 
\end{proof}

Note that this bound can not be improved in full generality. For example,
take the
sets $A=\{1,w,\dots,w^{n-1}\}$, $B=\{1,w,\dots,w^{n-2}\}$, where 
$w^{2n-4}=1$, and $w$ is primitive root of 1 of power $2n-4$. 
Then $|A|=n$, $|B|=n-1$, $|A\starr B|=2n-4$
and 1 has unique representation $w^{n-1}\cdot w^{n-3}$.

\begin{corollary} Let  $A$ be a finite subset of the multiplicative group of a field.
Assume that some element $c\in A\starr A$ has exactly two (symmetric to
each other) representations
as $c=ab=ba$, $a,b\in A$, $a\ne b$. Then  $|A\starr A|\geqslant 2|A|-4$.
\end{corollary}

\begin{proof}
Take $B=A\setminus b$ and apply Theorem \ref{multiplicative}
\end{proof}

This partially support conjecture of V. Lev \cite{Lev}.

Again example of $A=\{1,w,\dots,w^{n-1}\}$, $w^{2n-4}=1$ proves that
this bound is tight.

However, this may be improved in particular case $a^{n-2}\ne b^{n-2}$, where
$n=|A|$
(which just does not hold in our example):

\begin{theorem}\label{main}
Let  $A$, be a finite subset of the multiplicative group of a field, $|A|=n$.
Assume that some element $c\in A\starr A$ has exactly two (symmetric to
each other) representations
as $c=ab=ba$, $a,b\in A$, $a\ne b$. Assume also that
$a^{n-2}\ne b^{n-2}$. Then  $|A\starr A|\geqslant 2n-3$.
\end{theorem}

\begin{proof}
Assume that
$|A\starr A|\leqslant 2n-4$. A polynomial 
$$f(x,y):=(xy-1)\prod_{\gamma\in (A\starr A)\setminus c} (x-\gamma y)$$ of degree at most $2n-3$
vanishes on all points of $A\times A^{-1}$ except $(a,b^{-1})$ and $(b,a^{-1})$.
It allows to calculate its 
coefficient of $x^{n-1}y^{n-1}$ by the formula \eqref{formula}:
$$
[x^{n-1}y^{n-1}]f(x,y)=\sum_{t\in A,s\in A^{-1}} \frac{f(t,s)}{\prod_{\tau\in A\setminus t}(t-\tau) 
\prod_{\xi\in A^{-1}\setminus s}(s-\xi)}.
$$

We have two non-zero summands corresponding to points $(t,s)=(a,b^{-1})$
and $(t,s)=(b,a^{-1}).$ The first of them equals
\begin{align*}
\frac{(ab^{-1}-1)\prod_{\gamma\in (A\starr A)\setminus c} (a-\gamma b^{-1})}
{\prod_{\tau\in A\setminus a}(a-\tau) \prod_{\xi\in A^{-1}\setminus b^{-1}}(b^{-1}-\xi)}=\\
(a-b)b^{4-2n}\frac{\prod_{\gamma\in (A\starr A)\setminus c} (ab-\gamma)}
{(-1)^{n-1}\prod_{\tau\in A\setminus a}(a-\tau) \prod_{\xi\in A^{-1}\setminus b^{-1}}b^{-1}\xi(b-\xi^{-1})}=\\
(a-b)b^{2-n}(-1)^{n-1}\frac{\prod_{\gamma\in (A\starr A)\setminus c} (ab-\gamma)}
{\prod_{\tau\in A\setminus a}(a-\tau) \prod_{\xi\in A^{-1}\setminus b^{-1}}(b-\xi^{-1})
\prod_{\xi \in A^{-1}} \xi.
}
\end{align*}

The second summand has analogous expression, just change $a$ and $b$.
We see that their sum is not equal to 0 provided
that $a^{n-2}\ne b^{n-2}$.
\end{proof}

Lev posed an interesting question on estimating $|A\starr B|$ provided that
$A\starr B\ne A\times B$. What we managed to prove in this direction
(again for multiplicative groups of the field)  is in most cases
weaker than his conjecture:

\begin{theorem} Denote $N=\{a\in A\cap B:a^2\notin A\starr B\}$. Then
$|A\starr B|\geq |A|+|B|-2-[N/2]$.
\end{theorem}

\begin{proof} For any $\gamma\in A\starr B$ consider a line $x=\gamma y$.
They cover all points of $A\times B^{-1}$ except $N$ points $(a,a^{-1})$
on hyperbola $xy=1$. Add $[N/2]$ lines covering all those points except 1
(this is clearly possible) and apply Alon--F\"uredi Theorem \ref{AlFu}.
\end{proof}

I am deeply grateful to Gyula K\'arolyi and Vsevolod Lev for fruitful
discussions.


\begin{thebibliography}{99}

\bibitem{Lev} V. Lev, Restricted set addition in abelian groups: results and conjectures, 
J. Th. Nomb. Bordeaux 17 (1) (2005), 181--193.

\bibitem{L} M. Laso\'n, 
A generalization of Combinatorial Nullstellensatz,
Electron. J. Combin. 17 (2010) \#N32, 6 pages.

\bibitem{K} G. K\'arolyi.
The Erd\"os--Heilbronn problem in Abelian groups. Isr. J. Math. 139 (1) (2004),
249-359.

\bibitem{KP} R. N. Karasev, F. V. Petrov, 
Partitions of nonzero elements of a finite field into pairs, 
Israel J. Math. 192 (2012), 143--156.	

\bibitem{S} P. Scherk, Distinct elements in a set of sums (solution to Problem 4466). 
American Math. Monthly 62 (1) (1955), 46--47.

\bibitem{K1} J. H. B. Kemperman, On complexes in a semiroup. Indag. Math. 18 (1956), 247--254.

\bibitem{K2} J. H. B. Kemperman,
On small sumsets in an abelian group.  Acta Math. 103 (1960), 63--88.

\bibitem{Alon} N. Alon, 
Combinatorial Nullstellensatz,
Combin. Probab. Comput. 8 (1999), 7--29.

\bibitem{AF} N. Alon and Z. F\"uredi. Covering the cube by affine hyperplanes. Eur. J. Comb. 14
(1993), 79--83.


\bibitem{PS} H. Pan, Z.-W. Sun. Restricted sumsets and a conjecture of Lev.
Isrlael J. Math.  154(1) (2006), 21--28





\end{thebibliography}
\end{document}